\documentclass[12pt, reqno]{amsart}

\usepackage{url}
\usepackage{amsfonts}
\usepackage{amsthm}
\usepackage[psamsfonts]{amssymb}
\usepackage[dvipdfmx]{graphicx}
\usepackage{ascmac}
\usepackage{amsmath}
\usepackage{amsthm}
\usepackage{fancybox}
\usepackage{fancyhdr}
\usepackage{mathrsfs}
\usepackage{color}
\usepackage{multicol}
\usepackage{framed}

\theoremstyle{plain}
\newtheorem{thm}[equation]{Theorem}

\newtheorem{cor}[equation]{Corollary}
\newtheorem{rem}[equation]{\it Remark}

\newtheorem{lem}[equation]{Lemma}

\newtheorem*{Acknowledgements}{\it Acknowledgements}

\address{Department of Mathematics, Graduate School of Science, Osaka University, 1-1 Machikaneyama, Toyonaka, Osaka 560-0043, JAPAN}
\email{h-tadano@cr.math.sci.osaka-u.ac.jp}

\address{Faculty of Mathematics, University of Santiago de Compostela, 15782 Santiago de Compostela, SPAIN}
\email{homare.tadano@usc.es}

\title{An upper diameter bound for compact Ricci solitons with applications to the Hitchin-Thorpe inequality}
\dedicatory{}
\author{Homare TADANO}
\date{April 15, 2015}
\thanks{This work was supported by Moriyasu Graduate Student Scholarship Foundation}

\begin{document}

\begin{abstract}
In this article, stimulated by Fern\'{a}ndez-L\'{o}pez and Garc\'{i}a-R\'{i}o, we shall give an upper diameter bound for compact Ricci solitons in terms of the range of the scalar curvature. As an application, we shall provide some sufficient conditions for four-dimensional compact Ricci solitons to satisfy the Hitchin-Thorpe inequality.
\end{abstract}

\maketitle

\numberwithin{equation}{section}

\section{Introduction}

A \textit{Ricci soliton} \cite{H} is a complete Riemannian manifold $(M, g)$ admitting a smooth vector field $X \in \mathfrak{X}(M)$ such that
\begin{equation}\label{RS}
\operatorname{Ric}_{g} + \frac{1}{2} \mathcal{L}_{X} g = \lambda g
\end{equation}
for some real constant $\lambda \in \mathbb{R}$, where $\operatorname{Ric}_{g}$ denotes the Ricci tensor of $(M, g)$ and $\mathcal{L}_{X}$ is the Lie derivative in the direction of $X$. The soliton $(M, g)$ is said to be \textit{shrinking}, \textit{steady} and \textit{expanding} if $\lambda > 0, \lambda = 0$ and $\lambda < 0$, respectively. Typical examples of the Ricci soliton are Einstein manifolds, where $X$ is given by a Killing vector field. In this case, we say that the soliton is \textit{trivial}. Ricci solitons play an important role in the Ricci flow as they correspond to self-similar solutions and often arise as singularity models \cite{Cao2}. When $X$ may be replaced by the gradient $\nabla f$ for some smooth function $f : M \rightarrow \mathbb{R}$, called a \textit{potential function}, $(M, g)$ is called a \textit{gradient Ricci soliton}. In such a case, (\ref{RS}) becomes
\begin{equation}\label{GRS}
\operatorname{Ric}_{g} + H_{f} = \lambda g, 
\end{equation}
where $H_{f}$ denotes the Hessian of the function $f$. Due to Perelman \cite{P}, any compact Ricci soliton is a gradient one. It is well-known \cite{Cao2} that compact steady and expanding Ricci solitons must be trivial, as well as compact shrinking Ricci solitons in dimension two and three \cite{Cao2}. Examples of non-trivial compact K\"{a}hler-Ricci solitons were constructed by Koiso \cite{K}, Cao \cite{Cao1} and Wang and Zhu \cite{W-Z}.

A lower diameter bound for compact shrinking Ricci solitons has been recently investigated by many authors \cite{Andrews-Ni, Chu-Hu, FL-GR1, Futaki-Li-Li, Futaki-Sano}. In particular, a universal lower bound for compact shrinking Ricci solitons was first given by Futaki and Sano \cite{Futaki-Sano} in relation to study of the first non-zero eigenvalue of the Witten-Laplacian. On the other hand, Fern\'{a}ndez-L\'{o}pez and Garc\'{i}a-R\'{i}o \cite{FL-GR1} gave the following lower diameter bound in terms of the Ricci curvature and the range of the potential function.

\begin{thm}[Fern\'{a}ndez-L\'{o}pez and Garc\'{i}a-R\'{i}o \cite{FL-GR1}]\label{thm-intro-1}
Let $(M, g)$ be an $n$-dimensional compact connected shrinking Ricci soliton satisfying {\rm (\ref{GRS})}. Then
\[
\operatorname{diam}(M, g) \geqslant \max \left \{ \sqrt{\frac{2(f_{\mathrm{max}} - f_{\mathrm{min}})}{C - \lambda}}, \sqrt{\frac{2(f_{\mathrm{max}} - f_{\mathrm{min}})}{\lambda - c}}, 2 \sqrt{\frac{2(f_{\mathrm{max}} - f_{\mathrm{min}})}{C - c}} \right \}, 
\]
where $f_{\mathrm{max}}$ and $f_{\mathrm{min}}$ respectively denote the maximum and the minimum value of the potential function on the soliton.
\end{thm}

In Theorem \ref{thm-intro-1}, the number
\[
C : = \max_{v \in TM} \{ \operatorname{Ric}_{g}(v, v) : | v | = 1 \} \quad \mbox{and} \quad c : = \min_{v \in TM} \{ \operatorname{Ric}_{g}(v, v) : | v | = 1 \}
\]
respectively denote the maximum and the minimum value of the Ricci curvature on the unit sphere bundle of $(M, g)$. Note that $cg \leqslant \operatorname{Ric}_{g} \leqslant Cg$.

When the soliton has positive Ricci curvature, this diameter bound can be written in terms of the range of the scalar curvature as follows.

\begin{cor}[Fern\'{a}ndez-L\'{o}pez and Garc\'{i}a-R\'{i}o \cite{FL-GR1}]\label{cor-intro-2}
Let $(M, g)$ be an $n$-dimensional compact connected shrinking Ricci soliton with positive Ricci curvature satisfying {\rm (\ref{GRS})}. Then
\[
\operatorname{diam}(M, g) \geqslant \max \left \{ \sqrt{\frac{R_{\mathrm{max}} - R_{\mathrm{min}}}{\lambda(C - \lambda)}}, \sqrt{\frac{R_{\mathrm{max}} - R_{\mathrm{min}}}{\lambda(\lambda - c)}}, 2 \sqrt{\frac{R_{\mathrm{max}} - R_{\mathrm{min}}}{\lambda(C - c)}} \right \}, 
\]
where $R_{\mathrm{max}}$ and $R_{\mathrm{min}}$ respectively denote the maximum and the minimum value of the scalar curvature on the soliton.
\end{cor}

Moreover, stimulated by the Myers diameter estimate \cite[Theorem 1.4]{W-W} via Bakry-\'{E}mery Ricci curvature, Fern\'{a}ndez-L\'{o}pez and Garc\'{i}a-R\'{i}o \cite{FL-GR1} mentioned that an upper diameter bound for compact shrinking Ricci solitons would be given in terms of the range of the potential function, as well as in terms of the range of the scalar curvature.

The aim of this article is to give a positive answer to this conjecture by giving the following.

\begin{thm}\label{Main-Theorem}
Let $(M, g)$ be an $n$-dimensional compact connected shrinking Ricci soliton satisfying {\rm (\ref{GRS})}. Then
\begin{equation}\label{diam-Main-Theorem}
\operatorname{diam}(M, g) \leqslant \frac{1}{\lambda} \left( 2 \sqrt{R_{\mathrm{max}} - R_{\mathrm{min}}} + \sqrt{4(R_{\mathrm{max}} - R_{\mathrm{min}}) + (n - 1) \lambda \pi^{2}} \right).
\end{equation}
\end{thm}

When the soliton has positive Ricci curvature, this diameter bound can be written in terms of the range of the potential function as follows.

\begin{cor}\label{Main-Corollary}
Let $(M, g)$ be an $n$-dimensional compact connected shrinking Ricci soliton with positive Ricci curvature satisfying {\rm (\ref{GRS})}. Then
\[
\operatorname{diam}(M, g) \leqslant 2 \sqrt{\frac{2(f_{\mathrm{max}} - f_{\mathrm{min}})}{\lambda}} + \sqrt{\frac{8(f_{\mathrm{max}} - f_{\mathrm{min}}) + (n - 1) \pi^{2}}{\lambda}}.
\]
\end{cor}

Just as in the Einstein manifold, we may expect some topological obstruction to the existence of compact Ricci solitons. A validity of the Hitchin-Thorpe inequality for four-dimensional compact shrinking Ricci solitons was first shown by Ma \cite{Ma} assuming some upper bounds on the $L^{2}$-norm of the scalar curvature. On the other hand, Fern\'{a}ndez-L\'{o}pez and Garc\'{i}a-R\'{i}o \cite{FL-GR2} investigated the same validity assuming the following upper diameter bounds in terms of the Ricci curvature.

\begin{thm}[Fern\'{a}ndez-L\'{o}pez and Garc\'{i}a-R\'{i}o \cite{FL-GR1}]
Let $(M, g)$ be a four-dimensional compact connected shrinking Ricci soliton satisfying {\rm (\ref{GRS})}. If
\[
\operatorname{diam}(M, g) \leqslant \max \left \{ \sqrt{\frac{2}{C - \lambda}}, \sqrt{\frac{2}{\lambda - c}}, 2 \sqrt{\frac{2}{C - c}} \right \}, 
\]
then the soliton satisfies the Hitchin-Thorpe inequality $2 \chi(M) \geqslant 3 | \tau(M) |$.
\end{thm}

The following corollary of Theorem \ref{Main-Theorem} provides a sufficient condition for four-dimensional compact shrinking Ricci solitons to satisfy the Hitchin-Thorpe inequality.

\begin{cor}\label{Cor-1}
Let $(M, g)$ be a four-dimensional compact connected shrinking Ricci soliton satisfying {\rm (\ref{GRS})}. If
\begin{equation}\label{diam-Cor-1}
\sqrt{\frac{R_{\mathrm{max}} - R_{\mathrm{min}}}{\lambda^{2}}(16 + 6 \pi^{2})} \leqslant \operatorname{diam}(M, g), 
\end{equation}
then the soliton satisfies the Hitchin-Thorpe inequality $2 \chi(M) \geqslant 3 | \tau(M) |$.
\end{cor}

This note is organized as follows: In Section 2, after introducing our notation, we shall prove Theorem \ref{Main-Theorem} and Corollary \ref{Main-Corollary}. Ending with Section 3, a proof of Corollary \ref{Cor-1} and some related result will be given.

\begin{Acknowledgements}\rm
I thank Professor Toshiki Mabuchi for his encouragements. This work was carried out while the author was visiting University of Santiago de Compostela. I also thank Professor Eduardo Garc\'{i}a-R\'{i}o for his warm hospitality.
\end{Acknowledgements}

\section{Preliminaries}

In this section, after introducing our notation, we shall prove Theorem \ref{Main-Theorem} and Corollary \ref{Main-Corollary}. Let $X, Y, Z \in \mathfrak{X}(M)$ be three vector fields on $M$. For any smooth function $f \in \mathcal{C}^{\infty}(M)$, the gradient vector field and Hessian of $f$ are defined by
\[
g( \nabla f, X) = df(X) \quad \mbox{and} \quad H_{f}(X, Y) = g(\nabla_{X} \nabla f, Y), 
\]
respectively. The curvature tensor and Ricci tensor are defined by
\[
R(X, Y)Z = \nabla_{X} \nabla_{Y} Z - \nabla_{Y} \nabla_{X} Z - \nabla_{[X, Y]} Z \quad \mbox{and} \quad \operatorname{Ric}(X, Y) = \sum_{i = 1}^{n} g(R(e_{i}, X)Y, e_{i}), 
\]
respectively. Here, $\{ e_{i} \}_{i = 1}^{n}$ is an orthonormal frame of $(M, g)$. In order to prove Theorem \ref{Main-Theorem}, we will use the index form of a minimizing unit speed geodesic segment. We refer the reader to books \cite{Lee, Petersen} for basic facts about this topic. The following Myers type theorem plays an important role in proving Theorem \ref{Main-Theorem}.

\begin{thm}\label{thm-1}
Let $(M, g)$ be an $n$-dimensional complete connected Riemannian manifold. Suppose that $(M, g)$ admits a smooth vector field $V$ satisfying
\begin{equation}\label{assumption}
\operatorname{Ric}_{g} + \mathcal{L}_{V} g \geqslant (n - 1)C g \quad and \quad | V | \leqslant \gamma
\end{equation}
for some constants $C > 0$ and $\gamma \geqslant 0$. Then $(M, g)$ is compact and the diameter of $(M, g)$ has the upper bound
\begin{equation}\label{diam-result}
\operatorname{diam}(M, g) \leqslant \frac{4 \gamma + \sqrt{16 \gamma^{2} + (n - 1)^{2} C \pi^{2}}}{(n - 1)C}.
\end{equation}
\end{thm}

\begin{rem}\rm
Under the same condition as in Theorem \ref{thm-1}, Limoncu \cite{L} gave the following diameter estimate
\begin{equation}\label{diam-L}
\operatorname{diam}(M, g) \leqslant \frac{\pi}{(n - 1)C} \left( \sqrt{2} \gamma + \sqrt{2 \gamma^{2} + (n - 1)^{2}C} \right).
\end{equation}
Since
\[
\sqrt{2} \pi \approx 4.44288 > 4 \quad \mbox{and} \quad 2 \pi^{2} \approx 19.73920 > 16, 
\]
our diameter estimate (\ref{diam-result}) is sharper than (\ref{diam-L}).
\end{rem}

\begin{proof}[Proof of Theorem {\rm \ref{thm-1}}]
Our proof of Theorem \ref{thm-1} is similar to that by Limoncu \cite{L}. Take arbitrary two points $p, q \in M$. By the compactness of the manifold $(M, g)$, there exists the minimizing unit speed geodesic segment $\sigma$ from $p$ to $q$ of length $\ell$. Let $\{ e_{1} = \dot{\sigma}, e_{2}, \cdots, e_{n} \}$ be a parallel orthonormal frame along $\sigma$. Recall that, for any smooth function $\phi \in \mathcal{C}^{\infty}([0, \ell])$ satisfying $\phi(0) = \phi(\ell) = 0$, we obtain
\begin{equation}\label{Index}
I(\phi e_{i}, \phi e_{i}) = \int_{0}^{\ell} \left( g(\dot{\phi} e_{i}, \dot{\phi} e_{i}) - g(R(\phi e_{i}, \dot{\sigma}) \dot{\sigma}, \phi e_{i}) \right) dt, 
\end{equation}
where $I(\cdot, \cdot)$ denotes the index form of $\sigma$. From (\ref{Index}), we have
\begin{equation}\label{Index-sum}
\sum_{i = 2}^{n} I(\phi e_{i}, \phi e_{i}) = \int_{0}^{\ell} \left( (n - 1) \dot{\phi}^{2} - \phi^{2} \operatorname{Ric}_{g}(\dot{\sigma}, \dot{\sigma}) \right) dt, 
\end{equation}
where we have used $g(R(\dot{\sigma}, \dot{\sigma})\dot{\sigma}, \dot{\sigma}) = 0$. By using the assumption (\ref{assumption}) in the integral expression (\ref{Index-sum}), we obtain
\begin{align}
\sum_{i = 2}^{n} I(\phi e_{i}, \phi e_{i}) & \leqslant \int_{0}^{\ell} \left( (n - 1)(\dot{\phi}^{2} - C \phi^{2}) + \phi^{2} (\mathcal{L}_{V} g)(\dot{\sigma}, \dot{\sigma}) \right) dt \nonumber \\
& = \int_{0}^{\ell} \left( (n - 1)(\dot{\phi}^{2} - C \phi^{2}) + 2 \phi^{2} g(\nabla_{\dot{\sigma}} V, \dot{\sigma}) \right) dt \nonumber \\
& = \int_{0}^{\ell} \left( (n - 1)(\dot{\phi}^{2} - C \phi^{2}) + 2 \phi^{2} \dot{\sigma}(g(V, \dot{\sigma})) \right) dt, \label{eq1}
\end{align}
where, the last equality follows from the parallelism of the metric $g$ and $\nabla_{\dot{\sigma}} \dot{\sigma} = 0$. On the geodesic segment $\sigma(t)$, we have
\begin{align}
2 \phi^{2} \dot{\sigma} \left( g(V, \dot{\sigma}) \right) & = 2 \phi^{2} \frac{d}{dt}(g(V, \dot{\sigma})) \nonumber \\
& = - 4 \phi \dot{\phi} g(V, \dot{\sigma}) + 2 \frac{d}{dt}(\phi^{2} g(V, \dot{\sigma}). \label{eq2}
\end{align}
Hence, by integrating both sides of (\ref{eq2}), we have
\begin{align}
\int_{0}^{\ell} 2 \phi^{2} \dot{\sigma}( g(V, \dot{\sigma})) dt & = \int_{0}^{\ell} - 4 \phi \dot{\phi} g(V, \dot{\sigma}) dt + \left[ 2 \phi^{2} g(V, \dot{\sigma}) \right]_{0}^{\ell} \nonumber \\
& = \int_{0}^{\ell} - 4 \phi \dot{\phi} g(V, \dot{\sigma}) dt \label{eq3} \\
& \leqslant 4 \int_{0}^{\ell} \left| \phi \dot{\phi} g(V, \dot{\sigma}) \right| dt, \label{eq4}
\end{align}
where, the second equality follows from $\phi(0) = \phi(\ell) = 0$. Since $\sigma$ is a unit speed geodesic segment, the Cauchy-Schwarz inequality implies $| g(V, \dot{\sigma}) | \leqslant | V |$. By combining this inequality and the assumption $| V | \leqslant \gamma$ in Theorem \ref{thm-1}, we have $| g(V, \dot{\sigma}) | \leqslant \gamma$. Hence, from (\ref{eq4}) we obtain
\begin{equation}\label{eq5}
\int_{0}^{\ell} 2 \phi^{2} \dot{\sigma}( g(V, \dot{\sigma})) dt \leqslant 4 \gamma \int_{0}^{\ell} \left| \phi \dot{\phi} \right| dt.
\end{equation}
From (\ref{eq1}) and (\ref{eq5}), we have
\begin{equation}\label{eq-main}
\sum_{i = 2}^{n} I(\phi e_{i}, \phi e_{i}) \leqslant \int_{0}^{\ell} (n - 1)(\dot{\phi}^{2} - C \phi^{2}) dt + 4 \gamma \int_{0}^{\ell} \left| \phi \dot{\phi} \right| dt.
\end{equation}
If the funtion $\phi$ is taken to be $\phi(t) = \sin (\frac{\pi t}{\ell})$, then we obtain $\dot{\phi}(t) = \frac{\pi}{\ell} \cos (\frac{\pi t}{\ell})$ and
\[
\phi \dot{\phi} = \frac{\pi}{\ell} \sin \left( \frac{\pi t}{\ell} \right) \cos \left( \frac{\pi t}{\ell} \right) = \frac{\pi}{2 \ell} \sin \left( \frac{2 \pi t}{\ell} \right).
\]
Then, (\ref{eq-main}) becomes
\[
\begin{aligned}
\sum_{i = 2}^{n} I(\phi e_{i}, \phi e_{i}) & \leqslant (n - 1) \int_{0}^{\ell} \left( \frac{\pi^{2}}{\ell^{2}} \cos^{2} \left( \frac{\pi t}{\ell} \right) - C \sin^{2} \left( \frac{\pi t}{\ell} \right) \right) dt \\
& \quad + \frac{2 \gamma \pi}{\ell} \int_{0}^{\ell} \left| \sin \frac{2 \pi t}{\ell} \right| dt, 
\end{aligned}
\]
and consequently, we have
\[
\sum_{i = 2}^{n} I(\phi e_{i}, \phi e_{i}) \leqslant - \frac{1}{2 \ell} \left( (n - 1) C \ell^{2} - 8 \gamma \ell - (n - 1) \pi^{2} \right).
\]
Since $\sigma$ is a minimizing geodesic, we must obtain
\[
(n - 1)C \ell^{2} - 8 \gamma \ell - (n - 1) \pi^{2} \leqslant 0, 
\]
from where, we have
\[
\ell \leqslant \frac{4 \gamma + \sqrt{16 \gamma^{2} + (n - 1)^{2} C \pi^{2}}}{(n - 1)C}.
\]
This proves Theorem \ref{thm-1}.
\end{proof}

\begin{rem}\rm
Using Cauchy-Schwarz inequality, Limoncu estimated (\ref{eq3}) from above by
\[
\int_{0}^{\ell} 2 \phi^{2} \dot{\sigma}( g(V, \dot{\sigma})) dt = \int_{0}^{\ell} - 4 \phi \dot{\phi} g(V, \dot{\sigma}) dt \leqslant 4 \sqrt{\int_{0}^{\ell} (\phi \dot{\phi})^{2} dt} \sqrt{\int_{0}^{\ell} (g(V, \dot{\sigma}))^{2} dt}, 
\]
while we estimated (\ref{eq3}) from above by an absolute value in (\ref{eq4}) and obtained a better estimate (\ref{diam-result}) than (\ref{diam-L}).
\end{rem}

The following lemma is useful to prove Theorem \ref{Main-Theorem}.

\begin{lem}[Fern\'{a}ndez-L\'{o}pez and Garc\'{i}a-R\'{i}o \cite{FL-GR2}]\label{lem2}
Let $(M, g)$ be an $n$-dimensional compact shrinking Ricci soliton satisfying {\rm (\ref{GRS})}. Then
\begin{equation}\label{lem-ineq}
| \nabla f |^{2} \leqslant R_{\mathrm{max}} - R, 
\end{equation}
where $R$ denotes the scalar curvetre on the soliton.
\end{lem}

\begin{proof}
We recall the proof for the reader's convenience. It is well-known that the potential function $f$ of any gradient Ricci soliton $(M, g)$ satisfies
\begin{equation}\label{fund-eq2}
R + | \nabla f |^{2} - 2 \lambda f = C
\end{equation}
for some constant $C$, where $R$ denotes the scalar curvature on the soliton. By compactness of the manifold $M$, there exists some global maximum point $p \in M$ of the potential function. Then, it follows from (\ref{fund-eq2}) that, for any point $x \in M$, 
\begin{equation}\label{lem-eq}
2 \lambda f(p) = R(p) - C \geqslant 2 \lambda f(x) = R(x) + | \nabla f |^{2} (x) - C, 
\end{equation}
and hence, $R(p) \geqslant R(x)$. Therefore, the scalar curvature also attains its maximum at $p$, and we obtain (\ref{lem-ineq}).
\end{proof}

Now, we are in a position to prove Theorem \ref{Main-Theorem}.

\begin{proof}[Proof of Theorem {\rm \ref{Main-Theorem}}]
We apply Theorem \ref{thm-1} to the case that $(M, g)$ is a compact gradient shrinking Ricci soliton. From (\ref{lem-ineq}), we have $| \nabla f | \leqslant \sqrt{R_{\mathrm{max}} - R_{\mathrm{min}}}$. Hence, by applying
\[
V = \frac{1}{2} \nabla f, \quad C = \frac{\lambda}{n - 1} \quad \mbox{and} \quad \gamma = \frac{1}{2} \sqrt{R_{\mathrm{max}} - R_{\mathrm{min}}}
\]
to (\ref{diam-result}), we obtain (\ref{diam-Main-Theorem}).
\end{proof}

\begin{proof}[Proof of Corollary {\rm \ref{Main-Corollary}}]
We show that
\begin{equation}\label{Main-Corollary-eq}
2 \lambda f_{\mathrm{max}} - 2 \lambda f_{\mathrm{min}} = R_{\mathrm{max}} - R_{\mathrm{min}}.
\end{equation}
Although this equality was already proved by Fern\'{a}ndez-L\'{o}pez and Garc\'{i}a-R\'{i}o in \cite{FL-GR1}, we here show it for the reader's convenience. By using (\ref{fund-eq2}) and (\ref{GRS}), we obtain
\begin{equation}\label{fund-eq3}
\operatorname{Ric}_{g}(\nabla f, \cdot) = \frac{1}{2} dR.
\end{equation}
By compactness of the manifold $M$, there exists some global minimum point $q \in M$ of the scalar curvature. From (\ref{fund-eq3}), we have $0 = (\nabla R)(q) = 2 \operatorname{Ric}_{g}(\nabla f, \cdot)(q)$. Since $(M, g)$ has positive Ricci curvature, we have $(\nabla f)(q) = 0$. Then, it follows from (\ref{fund-eq2}) that, for any $x \in M$, 
\[
\begin{aligned}
R(q) = 2 \lambda f(q) - | \nabla f |^{2}(q) + C & = 2 \lambda f(q) + C \\
& \leqslant R(x) = 2 \lambda f(x) - | \nabla f |^{2}(x) + C \leqslant 2 \lambda f(x) + C, 
\end{aligned}
\]
from where we see that $q \in M$ is also a global minimum of the potential function, and hence, $R_{\mathrm{min}} = 2 \lambda f_{\mathrm{min}} + C$. On the other hand, we have shown in (\ref{lem-eq}) that $R_{\mathrm{max}} = 2 \lambda f_{\mathrm{max}} + C$. Therefore, we obtain (\ref{Main-Corollary-eq}). Corollary \ref{Main-Corollary} follows immediately from Theorem \ref{Main-Theorem} and (\ref{Main-Corollary-eq}).
\end{proof}

\section{Applications to Theorem {\rm \ref{Main-Theorem}}}

In this section, by using Theorem \ref{Main-Theorem}, we shall give a proof of Corollary \ref{Cor-1}. Throughout this section, we assume that $(M, g)$ is a compact connected shrinking Ricci soliton satisfying (\ref{GRS}). We use the following theorem to prove Corollary \ref{Cor-1}.

\begin{thm}[Ma \cite{Ma}]\label{Ma}
Let $(M, g)$ be a four-dimensional compact shrinking Ricci soliton satisfying {\rm (\ref{GRS})}. If the scalar curvature satisfies
\[
\int_{M} R^{2} \leqslant 24 \lambda^{2} \mathrm{vol}(M, g), 
\]
then the soliton $(M, g)$ satisfies the Hitchin-Thorpe inequality $2 \chi(M) \geqslant 3 | \tau(M) |$.
\end{thm}

\begin{proof}[Proof of Corollary {\rm \ref{Cor-1} }]
By taking the trace of (\ref{GRS}), we have
\begin{equation}\label{Cor-1-eq1}
R + \Delta f = 4 \lambda.
\end{equation}
Thanks to Theorem \ref{Main-Theorem}, the diameter of $(M, g)$ has the upper bound
\begin{equation}\label{diam-Cor-1-eq2}
\operatorname{diam}(M, g) < \frac{2}{\lambda} \sqrt{4(R_{\mathrm{max}} - R_{\mathrm{min}}) + 3 \lambda \pi^{2}}.
\end{equation}
Suppose that the inequality (\ref{diam-Cor-1}) holds. Then, from (\ref{diam-Cor-1-eq2}), we obtain
\[
\frac{R_{\mathrm{max}} - R_{\mathrm{min}}}{\lambda^{2}}(16 + 6 \pi^{2}) \leqslant \operatorname{diam}^{2}(M, g) < \frac{4}{\lambda^{2}} \left \{ 4(R_{\mathrm{max}} - R_{\mathrm{min}}) + 3 \lambda \pi^{2} \right \}, 
\]
from where we have $R_{\mathrm{max}} < 6 \lambda$. Hence, by (\ref{Cor-1-eq1}), we have
\[
\int_{M} R^{2} \leqslant R_{\mathrm{max}} \int_{M} R < 24 \lambda^{2} \mathrm{vol}(M, g), 
\]
and the result follows from Theorem \ref{Ma}.
\end{proof}

The following result follows immediately from Theorem \ref{Main-Theorem} and Theorem \ref{Ma}.

\begin{cor}\label{Cor-2}
Let $(M, g)$ be a four-dimensional compact connected shrinking Ricci soliton satisfying {\rm (\ref{GRS})}. If
\[
\frac{R_{\mathrm{max}}}{6 \lambda} \cdot \frac{1}{\lambda} \left( 2 \sqrt{R_{\mathrm{max}} - R_{\mathrm{min}}} + \sqrt{4(R_{\mathrm{max}} - R_{\mathrm{min}}) + 3 \lambda \pi^{2}} \right) \leqslant \operatorname{diam}(M, g), 
\]
then the soliton satisfies the Hitchin-Thorpe inequality $2 \chi(M) \geqslant 3 | \tau(M) |$.
\end{cor}

\end{document}